\documentclass{amsart}

\usepackage{amsthm,amssymb,verbatim}
\usepackage{graphicx}
\usepackage{enumerate}

\newcommand{\vv}{\bold v}
\newcommand{\ee}{\bold e}
\newcommand{\uu}{\bold u}

 \newcommand{\RR}{\mathbf{R}}  % reals
 \newcommand{\dist}{\operatorname{dist}}
 
 \newcommand{\area}{\operatorname{area}}
 \newcommand{\eps}{\epsilon}
 \newcommand{\Tan}{\operatorname{Tan}}
 
 \newcommand{\ddt}{\frac{d}{dt}}
  \newcommand{\Ddt}{\left(\frac{d}{dt}\right)_{(t=0)}}

 \newcommand{\trace}{\operatorname{trace}}

\newcommand{\spt}{\operatorname{spt}}

\input epsf
\def\begfig {
\begin{figure}
\small }
\def\endfig {
\normalsize
\end{figure}
}

    \newtheorem{theorem}    {Theorem}    %   [section]

    \theoremstyle{definition}

    \theoremstyle{definition}

\usepackage{hyperref}
\usepackage[alphabetic]{amsrefs}
\begin{document}

\renewcommand{\thesubsection}{\thetheorem}
   % this is so subsections and theorems, etc will be
   % numbered together.

\title[Maximum Principle]{The Maximum Principle for Minimal \\ Varieties of
    Arbitrary Codimension}
\author{Brian White}
\thanks{This research was supported by the NSF 
  under grant~DMS-0707126}
\begin{abstract} 
We prove  that
an $m$-dimensional minimal variety in a Riemannian
manifold cannot touch the boundary at a point where 
the sum of the smallest $m$ principal curvatures is greater than $0$.
We prove a stronger maximum principle in case the variety is a hypersurface.
We also prove analogous results for varieties with bounded mean curvature.

\end{abstract}
\email{white@math.stanford.edu}
\date{June 9, 2009. Revised November 11, 2010.}

%\begin{abstract}
%\end{abstract}

% 49 is "manifolds".
% 49Q20 is variational problems in a GMT setting.
% 49Q05 is "minimal surfaces".

\maketitle

%\section{Introduction}\label{section:intro}

Let $N$ be a smooth Riemannian manifold with boundary.  
In general, $N$ need
not be complete.   
Suppose $X$ is a compactly supported $C^1$ tangent vectorfield on $N$ such that
\begin{equation}\label{into}
  X\cdot \nu_N \ge 0
\end{equation}
at all points of $\partial N$, where $\nu_N$ is the unit normal to $\partial N$
that points into $N$.
Then $X$ generates a one-parameter family $t\in [0,\infty)\mapsto \phi_t$ of
maps of $N$ into itself such that $\phi_0$ is the identity map and such that
\[
   \ddt \phi_t(\cdot) = X(\phi_t(\cdot)).
\]
If $V$ is a $C^1$ submanifold of $N$ with finite area, we let $\delta V(X)$ denote
the first variation of area of $V$ with respect to $X$:
\[
  \delta V(X) = \Ddt \area(\phi_t(V)).
\]
More generally, if $V$ has locally finite area, we let
\[
  \delta V(X) = \Ddt \area(\phi_t(V\cap W))
\]
where $W$ is any open subset of $N$ that contains the support of $X$ and that
has compact closure.
Even more generally, $V$ can be any varifold in $N$.  (The theorems in this paper
are stated for arbitrary varifolds, but readers unfamiliar with varifolds may subsititute
``$C^1$ submanifold''  for ``varifold''  throughout the paper.
The appendix to~\cite{WhiteIsoperimetric} is a brief introduction to varifolds that
contains most of what is needed here.  For a more thorough treatment, 
see~\cite{SimonBook} or~\cite{AllardFirstVariation}.)

We say that a varifold $V$ in $N$ {\em minimizes area to first order} in $N$
provided
\[
   \delta V(X) \ge 0
\]
for every  compactly supported $C^1$ tangent vectorfield $X$ on $N$ satisfying~\eqref{into}.
In particular, any smooth minimal submanifold of $N$ or, more generally, any stationary varifold
in $N$ minimizes area to first order in $N$.

We say that $N$ is {\em strongly $m$-convex} at a point $p\in \partial N$ provided
\[
   \kappa_1 + \kappa_2 + \dots + \kappa_m  > 0
\]
where $\kappa_1\le \kappa_2 \le \dots \le \kappa_{n-1}$ are the principal curvatures of $\partial N$
at $p$ with respect to the unit normal $\nu_N$ that points into $N$.  We say that $N$ is $m$-convex
at $p$ provided $\kappa_1 + \kappa_2 + \dots + \kappa_m\ge 0$.

\begin{theorem}[Maximum Principle for Minimal Varifolds]\label{main}
Let $N$ be a smooth Riemannian manifold with boundary, and let $p$ be a point in $\partial N$
at which $N$ is strongly $m$-convex.
Then $p$ is not contained in the support of any $m$-dimensional varifold in $N$ that
minimizes area to first order in $N$.
Indeed, there is an $\eps>0$ such that
\[
   \dist(p, \spt V) \ge \eps
\]
for all such varifolds $V$.
\end{theorem}

In case $V$ is a smooth minimal submanifold, the fact that $V$ cannot contain $p$
was proved by Jorge and Tomi~\cite{JorgeTomi}.   Indeed, they proved that if $N$ is $m$-convex 
  (not necessarily strongly $m$-convex)
at all boundary points, then a smooth, connected minimal $m$-dimensional submanifold
cannot touch $\partial N$ unless it lies entirely in $\partial N$.
An analogous result for smooth submanifolds $V$ of bounded mean curvature was proved by Dierkes and Schwab~\cite{DierkesSchwab}
assuming (in addition to the appropriate condition on $\partial N$) that $N$ is flat.
The proofs here are similar to the proofs in those papers.  
The results here are stronger in that they apply to nonsmooth varieties (varifolds), and in that
the varieties are not assumed to be stationary with respect to all variations, but only with respect
to variations that take $N$ into itself (i.e., that satisfy~\eqref{into}).  This extra generality
is important because when one solves variational problems for area in a manifold-with-boundary $N$, 
the solution surfaces need not be stationary (or minimal), but rather only stationary 
with respect to variations satisfying~\eqref{into}.  Likewise, even if a solution surface turns out to be a
differentiable manifold, it need not be smooth -- it may only be $C^{1,1}$.  Of course if the boundary
of $N$ is strictly $m$-convex, then according to Theorem~\ref{main}, the solution surface does lie in the interior
of $N$, which implies that it is stationary with respect to all variations, and that if it is a differentiable
manifold, then it is a smooth submanifold.   But it is important that these
properties of the surface are conclusions of Theorem~\ref{main}
rather than hypotheses.

Theorem~\ref{main} is proved by constructing a suitable test vectorfield $X$:

\begin{theorem}\label{comparison}
Let $N$ be a smooth Riemannian manifold with boundary.
Let $p$ be a point in $\partial N$ and let $\eta < \kappa_1 + \dots + \kappa_m$,
where $\kappa_1 \le \dots \le \kappa_{n-1}$ are the principal curvatures
of $\partial N$ at $p$ with respect to the unit normal~$\nu_N$ that points into $N$.
Then there is a compactly supported $C^\infty$ vectorfield $X$ on $N$ 
such that $X(p)$ is a nonzero normal to $\partial N$, such that
\begin{equation}
   \text{ $X \cdot \nu_N \ge 0$ at all points of $\partial N$,} 
\end{equation}
and such that
\begin{equation}\label{h-integral}
  \delta V(X) \le - \eta \int |X|\, d\mu_V
\end{equation}
for every $m$-dimensional varifold $V$ in $N$.
\end{theorem}

Here $\mu_V$ is the weight measure associated to $V$.  (If $V$ is a $C^1$ submanifold,
then the integration in~\eqref{h-integral} is 
simply integration over $V$ with respect to $m$-dimensional area.)

We remark that $X$ can be chosen so that its support is contained in an arbitrarily small 
neighborhood of  $p$ and so that the vectorfield $X/ |X|$ (wherever $X$ is nonzero) is arbitrarily 
$C^0$-close to $\nabla \dist(\cdot, \partial N)$.  (In the proof below, one simply chooses
$\eps$ sufficiently small.)

To see that Theorem~\ref{main}
follows from Theorem~\ref{comparison}, note that if $N$ is strongly $m$-convex at $p$,
 then we may choose the $\eta$ in Theorem~\ref{comparison} to be positive.
If $V$ minimizes area to first order in $N$, then by definition and by Theorem~\ref{comparison},
\[
   0 \le \delta V(X) \le - \eta \int |X|\, d\mu_V.
\]
Since $\eta >0$, this implies that $|X|$ vanishes $\mu_V$-almost everywhere  
and thus that the support of $V$ cannot contain any point where $X\ne  0$. 
Hence $\dist(p, \spt V)\ge \eps$, where $\eps$ is the
 distance from $p$ to the nearest point where $X$ vanishes.

\begin{proof}[Proof of Theorem~\ref{comparison}]
Given a compactly supported $C^1$ vectorfield $X$ on $N$, let 
$\Psi_X: N\to \RR$ be the function 
\[
  \Psi_X(x) = \max \left(  \trace(\nabla X \vert P) \right)
\]
where the maximum is over all $m$-dimensional linear subspaces $P$
of $\Tan_xN$ and where
\[
    \trace(\nabla X \vert P) = \sum_{i=1}^m \uu_i\cdot \nabla_{\uu_i}X
\]
for any orthonormal basis $\uu_1, \dots, \uu_m$ of $P$. 

If $V$ is an $m$-dimensional $C^1$ submanifold of $N$, then
by the first variation 
formula~\cite{SimonBook}*{\S9.3}\footnote{
Equation~\eqref{FirstVariationFormula1} is proved by expressing the area of $\phi_t(V)$ as the integral
of a Jacobian determinant and then differentiating under the integral sign.
If $V$ is a smooth submanifold, one can then integrate by parts
to express $\delta V(X)$ as $-\int X\cdot H\,d\mu_V$.}
\begin{align}\label{FirstVariationFormula1}
   \delta V(X) &= \int \trace(\nabla X \vert \Tan_x V) \, d\mu_Vx.
   \\
   &\le \int \Psi_X \, d\mu_V. \notag
\end{align}
More generally, if $V$ is any $m$-dimensional varifold in $N$, then
by the first variation formula~\cite{SimonBook}*{\S 39.2},
\begin{align*}
 \delta V(X)
 &=
 \int_{(x,P) \in G_m(N)} \trace(\nabla X \vert P) \, dV(x,P)  \\
 &\le
 \int_{(x,P)\in G_m(N)} \Psi_X(x)\, dV(x,P) \\
 &=
 \int \Psi_X \, d\mu_V.
\end{align*}
where $G_m(N)$ is the set of pairs $(x,P)$ such that $x\in N$ and $P$ is
an $m$-dimensional linear subspace of $\Tan_xN$.

Thus we see that the conclusion~\eqref{h-integral} of Theorem~\ref{comparison} will hold provided
\begin{equation}\label{PsiBound}
   \Psi_X(\,\cdot\,) \le - \eta\, |X(\,\cdot\,)|
\end{equation}
at all points of $N$.

To construct the desired vectorfield $X$, we
may assume that $N$ is part of a larger Riemannian manifold $\tilde N$  (without boundary)
of the same dimension\footnote{
If the existence of a such a $\tilde N$ is not clear, note that $p$ has a neighborhood diffeomorphic to a closed half-space in $\RR^n$.
Since Theorem~\ref{comparison} is local, we can assume that $N$ is that half-space with some smooth Riemannian
metric.  We can extend the Riemannian metric to all of $\RR^n$ and then let $\tilde N$ be $\RR^n$ with
the extended metric.}.
Let 
\[
   \Sigma = \{ q\in \tilde N:  \dist(x, N) = \dist(x,p)^4 \}.
\]
Note that $\Sigma$ and $\partial N$ make second order contact at $p$.
By replacing $\tilde N$ with a small geodesic ball around $p$, we may assume
that $\Sigma$ is a smooth hypersurface and that there is smooth, well-defined
nearest-point retraction from $\tilde N$ to  $\Sigma$.  (We will later replace $\tilde N$
by an even smaller ball to ensure that additional conditions are satisfied.)

For $x\in \tilde N$, let $u(x)$ be the signed distance from $x$ to $\Sigma$, with
the sign chosen so that $u$ is nonnegative on $N$.
For $q\in \tilde N$, let $\Sigma_q$ be the level set of $u$ that contains $q$.
Note that $\nu(q):=\nabla u(q)$ is a unit normal to $\Sigma_q$.
Let 
\begin{equation}\label{kappa-order}
   k_1(q)\le \dots \le k_{n-1}(q)
\end{equation}
 be the principal curvatures
of $\Sigma_q$ at $q$ with respect to the unit normal $\nu(q)$.

Note that 
\begin{equation}\label{convexity}
    k_1 + \dots + k_m > \eta
\end{equation}
at $p$ since $\Sigma$ and $\partial N$ make
second order contact at $p$.
By replacing $\tilde N$ with a sufficiently small ball around $p$, we may assume
that~\eqref{convexity} holds at all points of $\tilde N$, 
that
\begin{equation}\label{normals}
   \text{  $\nu\cdot \nu_N >0$ at all points of $\partial N$},
\end{equation}
and that the $|k_i|$ are uniformly bounded:
\begin{equation}\label{Kbound}
   | k_i(q)| \le K \qquad (q\in \tilde N, \, i\le n-1).
\end{equation}

Let  $\eps$ be a positive number (to be specified later), and define a vectorfield $X$
on $N$ by
\[
   X(\cdot) = \phi(u(\cdot)) \, \nu(\cdot)
\]
where 
\[
   \phi(t)
   =
   \begin{cases}
   \exp\left(\frac1{t - \eps} \right) &\text{if $0\le t < \eps$}, \\
   0 &\text{if $t\ge \eps$.}
   \end{cases}
\]
(We need not define $\phi(t)$ for $t<0$ since $u\ge 0$ on $N$.)
Note that 
\[
   \frac{\phi'(t)}{\phi(t)} = \frac{-1}{(t  -\eps)^2} \le  \frac{-1}{\eps^2}
\]
for $0 \le t < \eps$, and thus
\begin{equation*}
   \phi'(t) \le -\frac{1}{\eps^2} \,\phi(t)
\end{equation*}
for all $t\ge 0$.   Thus by choosing $\eps \le  K^{-1/2}$, we can ensure that
\begin{equation}\label{phi-bound}
  \phi'(t) \le - K \phi(t)
\end{equation}
for all $t\ge 0$.

We also choose $\eps$ small enough that $N\cap \{u\le \eps\}$ is compact.

We claim that the vectorfield $X$ has the desired 
 properties.
  First note that 
\[
   \spt X = \overline{N\cap \{u<\eps\}},
\]
which is compact by choice of $\eps$.
Also, 
\[
    X \cdot \nu_N = \phi(u) \, \nu\cdot \nu_N  \ge 0
\]
at all points of $\partial N$ by~\eqref{normals}, since $\phi$ is everywhere nonnegative.

It remains only to show that $\Psi_X\le -\eta \, |X|$.
Let $q$ be any point in $N$.  
Let $\ee_1, \dots, \ee_{n-1}$ be principal
directions in $\Tan_q\Sigma_q$ corresponding
to the principal curvatures $k_1(q), \dots, k_{n-1}(q)$.
Consider the bilinear form $Q$ on $\Tan_qN$ given
by
\[
  Q(\uu,\vv) = \uu \cdot \nabla_{\vv} X.
\]
We wish to calculate the matrix for $Q$ with respect to the orthonormal
basis $\ee_1, \dots, \ee_{n-1}, \nu$.

Note that if $\vv$ is tangent to $\Sigma_q$, then
\begin{equation}\label{tangent}
   \nabla_\vv X = \nabla_\vv (\phi(u) \nu) = \phi(u) \nabla_\vv \nu
\end{equation}
since $\phi(u)$ is constant on $\Sigma_q$.
Thus if $\uu$ and $\vv$ are both tangent to $\Sigma_q$, then
\[
   Q(\uu,\vv) = \phi(u) \uu\cdot \nabla_\vv \nu =  - \phi(u) B(\uu,\vv)
\]
where $B$ is the second fundamental form of $\Sigma_q$ with respect to
the normal $\nu$.
In particular,
\[
   Q(\ee_i, \ee_j) = 
   \begin{cases}
   - \phi(u) k_i &\text{if $i=j$}, \\
   0 &\text{if $i\ne j$}.
   \end{cases}
\]
Since $\|\nu\|\equiv 1$, we see that $\nabla_\vv\nu$ is perpendicular to $\nu$
and thus
\[
  Q(\nu, \ee_i) = 0
\]
by~\eqref{tangent}.
Since $\nu$ is the gradient of the distance function,
$
   \nabla_\nu \nu  = 0
$.
Thus
\begin{align*}
  \nabla_\nu X 
  &= \nabla_\nu (\phi(u)\nu)  \\
  &= \phi'(u) (\nabla_\nu u) \nu + \phi(u) \nabla_\nu\nu \\
  &= \phi'(u) \nu,
\end{align*}
so $Q(\nu,\nu)= \phi'(u)$ and $Q(\ee_i, \nu)=0$.

Hence we see that the matrix for $Q$ with respect to the orthonormal
basis $\ee_1,\dots, \ee_{n-1}, \nu$ is a diagonal matrix
with diagonal elements $-\phi(u) k_i$ (for $1 \le i \le  n-1$)
and $\phi'(u)$.  Note that
\[
   -\phi(u) k_1 \ge -\phi(u)k_2 \ge \dots \ge -\phi(u) k_{n-1} \ge - \phi(u)K \ge \phi'(u)
\]
by~\eqref{kappa-order}, \eqref{Kbound}, and~\eqref{phi-bound}, 
since $\phi\ge 0$.  In particular, since $\phi(u)=|X|$, the largest $m$ eigenvalues
of $Q$ are $ -|X(q)|\, k_i$ where $1\le i \le m$.   It follows by elementary linear
algebra that
\begin{align*}
   \Psi(q) 
   &= \max_P \trace(Q \vert P) \\
   &= - |X(q)| \, (k_1(q) + \dots + k_m(q)) \\
   &\le - \eta \, |X(q)|
\end{align*}
by~\eqref{convexity}.  This completes the proof.
\end{proof}

\begin{comment}
We remark that the distance from $p$ to the nearest point
where $X$ vanishes is~$\eps$.   Thus Theorem~\ref{main} is true for that $\eps$.  (This assumes that we have
chosen the $\eta$ in Theorem~\ref{comparison} to be positive.  See the discussion after
Theorem~\ref{main}.)
\end{comment}

 \begin{theorem}[Maximum Principle for Set-Theoretic Limits of Minimal Varieties]\label{minimal-limits}
Suppose $N_i$ is a sequence of smooth Riemannian $n$-manifolds with boundary,
and suppose that the $N_i$ converge smoothly to a limit Riemannian manifold $N$.
Suppose for each $i$ that $V_i$ is an $m$-dimensional varifold in $N_i$ that minimizes area to first order in $N_i$, and
suppose that the sets $\spt(V_i)$ converge to a limit set $S \subset N$.
Then $S$ does not contain any point of $\partial N$ at which $N$ is strongly $m$-convex.
\end{theorem}

\begin{proof}
Since the result is local, we may assume that the $N_i$ and $N$ are all the same as smooth manifolds 
but have Riemannian metrics $g(i)$ and $g$ where $g(i)$ converges smoothly to $g$.

Let $\kappa_1(\cdot) \le \dots \le \kappa_{n-1}(\cdot)$ be the principal curvatures of $\partial N$ with
respect to the inward pointing unit normal.

Let $p$ be a point of $\partial N$ at which $N$ is strongly $m$-convex (with respect to $g$.)
Let $0< \eta < \kappa_1(p) + \dots + \kappa_m(p)$.
In the proof of Theorem~\ref{comparison}, we constructed a smooth function $u: N \to \RR$ with the following
properties (with respect to the metric $g$):
\begin{enumerate}[\upshape (i)]
\item $u(p)=0$ and $u>0$ on $N\setminus \{p\}$.
\item The set $C=\{u\le \eps\}$ is compact.
\item\label{boundaryconvexity} $\kappa_1 + \dots + \kappa_m > \eta$ at all points of $C \cap \partial N$.
\item\label{levelsetconvexity} $\nabla u$ never vanishes on $C$, and 
\[
    k_1(q) + \dots + k_m(q) > \eta
\]
at each point $q\in C$, where $k_1(q) \le \dots \le k_{n-1}(q)$ are the principal curvatures
of the level set $\Sigma_q = \{ x: u(x)=u(q)\}$  with respect to the unit normal $\nabla u(q)$.
\end{enumerate}

By the smooth convergence $g(i)\to g$, these properties will also hold with respect to the metric $g(i)$ for
all sufficiently large $i$.  
Fix such an $i$.   We claim that $\spt V_i$ cannot contain any point of $C$.  For if it did, the function
$u$ restricted to $C\cap \spt V_i$ would attain a minimum at some point $q$.
By~\eqref{boundaryconvexity} and by Theorem~\ref{main}, $q$ cannot be in $\partial N$.
By~\eqref{levelsetconvexity},  the set $\{u\ge u(q)\}$ is strongly $m$-convex at $q$, which contradicts Theorem~\ref{main}
(since $q\in \spt V_i \subset \{u\ge u(q)\}$.)
Thus $C\cap \spt V_i$ is empty.   Since $p$ is in the interior of $C$,
we are done.
\end{proof}

In the case of hypersurfaces, we get a stronger result:

\begin{theorem}[Strong Maximum Principle for Minimal Hypersurfaces]\label{SolomonWhite} 
Suppose that $N$ is a smooth Riemannian manifold (not necessarily complete) with boundary, 
that $\partial N$ is connected, and 
that $N$ is mean convex, i.e., that
\[
   H \cdot \nu_{N} \ge 0
\]
on $\partial N$, where $H$ is the mean curvature vector of $\partial N$ and where $\nu_N$ is the 
unit normal to $\partial N$ that points into $N$.   Let $m=\dim(N)-1$, and suppose that $V$ is an $m$-dimensional varifold that 
minimizes area to first order in $N$.  
\begin{enumerate}
\item\label{strong1} If $\spt V$ contains any point of $\partial N$, then it must contain all of $\partial N$
and $H$ must vanish everywhere on $\partial N$.  
\item\label{strong2}
If $V$ is a stationary integral varifold, then $V$ can be written as $W+W'$ where the support of $W$ is $\partial N$
and the support of $W'$ is disjoint from $\partial N$.
\end{enumerate}
\end{theorem}

\begin{proof}
Assertion~\eqref{strong1} was proved by Solomon and White~\cite{SolomonWhite}.
Assertion~\eqref{strong1} also follows rather directly from Theorem~\ref{main}: 
 see \cite{SolomonWhite}*{Step 1, p. 687} and the comments
at the end of~\cite{SolomonWhite}.  

To prove~\eqref{strong2}, 
we may assume that $\partial N$ is a minimal hypersurface. (Otherwise $\spt V$ is disjoint from $\partial N$ by 
assertion~\ref{strong1}, so we can let $W=0$ and $W'=V$.)
Let $d$ be the smallest integer such that there is a point $p\in \partial N$ at which the 
density of $V$ is $d$.  Let $W$ be the $m$-dimensional integral varifold whose support is $\partial N$ and whose
density is $d$ at every point of $\partial N$.
Then $\mu_W\le \mu_V$, so (since $V$ and $W$ are rectifiable varifolds) $W\le V$ (as measures
on the Grassman bundle.)   Thus the signed measure $W':=V-W$ is in fact a positive measure, i.e., a varifold.
Since $V$ and $W$ are stationary integral varifolds, so is $W'$.   By choice of $d$, the varifold $W'$ has
density $0$ at least one point $p$ of $\partial N$.  It follows that $p$ is not in the support of $W'$
(because the density is $\ge 1$ at every point in the support of a stationary integral varifold).  
But then by assertion~\eqref{strong1}, $\spt W'$ is disjoint from $\partial $N.
\end{proof}

Assertion~\eqref{strong2} need not hold if $V$ is not an integer-multiplicity varifold.   For example, let $N$ be
a closed half space in $\RR^3$, let $P_i$ ($i=1,2,3,\dots$) be a sequence of planes in the interior of $N$ that
converge to $\partial N$, let $V_i$ be the varifold corresponding to $P_i$ with multiplicity $2^{-i}$, and let $V$
be the sum of the $V_i$.

See~\cite{TomMax} and~\cite{Schaetzle} for other strong maximum principles for
varieties of codimension $1$.  In particular, \cite{TomMax} gives a very general strong maximum
principle for pairs of codimension $1$ minimal varieties, both of which may be singular.

\begin{theorem}[Maximum Principle for Varieties with Bounded Mean Curvature]\label{main-boundedversion}
Let $N$ be a smooth Riemannian manifold with boundary and $h$ be a nonnegative number.  Suppose $V$ is an $m$-dimensional 
varifold in $N$ and that
\begin{equation}\label{boundedmeancurvature}
   \delta V(X) +  h \int |X|\, d\mu_V \ge 0
\end{equation}
for every compactly supported $C^1$ vectorfield on $N$ such that
\begin{equation}\label{inward}
  \text{$X \cdot \nu_N \ge 0$ at all points of $\partial N$.}
\end{equation}
 Then the support of $V$ cannot contain any point $p$ in $\partial N$
at which 
\[
  \kappa_1 + \dots  + \kappa_m >   h
\]
where $\kappa_1\le \kappa_2 \le \dots \le \kappa_{n-1}$ are the principal curvatures of $\partial N$
with respect to the unit normal $\nu_N$ that points into $N$.

Indeed, there is an $\eps=\eps(h)$ such that $\dist(p, \spt V)\ge \eps$ for all $m$-varifolds $V$
satisfying~\eqref{boundedmeancurvature}.
\end{theorem}

To understand the meaning of the hypothesis on $V$, suppose that $V$ is a smooth $m$-dimenisonal
submanifold.   In that case, the inequality~\eqref{boundedmeancurvature} holds for all compactly supported
$C^1$ vectorfields if and only if the length of the mean curvature vector of $V$ is everywhere bounded by $h$.
The inequality holds for all $X$ satisfying~\eqref{inward} if and only if: (i) at every point in $V\setminus \partial N$,
the length of the mean curvature vector is at most $h$, and (ii) at every point $q\in V\cap \partial N$, the mean
curvature vector at $q$ can be written as the sum of a vector of length at most $h$ and a normal vector to $\partial N$
that points out of $N$.

Theorem~\ref{main-boundedversion} follows from Theorem~\ref{comparison} exactly as Theorem~\ref{main} did.  (One chooses the $\eta$
in Theorem~\ref{comparison} to be strictly between $h$ and $\kappa_1(p) + \dots + \kappa_m(p)$.)
Note that Theorem~\ref{main} is Theorem~\ref{main-boundedversion} in the special case $h=0$.

 \begin{theorem}[Maximum Principle for Set-Theoretic Limits of  Varieties with Bounded Mean Curvature]
Suppose that $N_i$ is a sequence of smooth Riemannian $n$-manifolds with boundary,
and suppose that the $N_i$ converge smoothly to a limit Riemannian manifold $N$.
Suppose for each $i$ that $V_i$ is an $m$-dimensional varifold in $N_i$ and that
\[
   \delta V_i(X) +  h \int |X|\, d\mu_{V_i} \ge 0  
\]
for every compactly supported $C^1$ vectorfield $X$ on $N_i$ 
such that
\[
  \text{ $X\cdot \nu_{N_i} \ge 0$ at all points of $\partial N_i$.}
\]
Suppose also that the sets $\spt(V_i)$ converge to a limit set $S \subset N$.
Then $S$ does not contain any point of $\partial N$ at which 
\[
    \kappa_1 + \dots + \kappa_m > h,
\]
where $\kappa_1 \le \kappa_2 \le \dots \le \kappa_{n-1}$ are the principal 
curvatures of $\partial N$ with respect to the unit normal that points into $N$.
\end{theorem}

The proof is almost identical to the proof of Theorem~\ref{minimal-limits}.

\begin{theorem}[Strong Maximum Principle for Hypersurfaces with Bounded Mean Curvature]\label{BoundedMCHypersurface}
Let $N$ be a smooth Riemannian manifold with boundary. Suppose that $\partial N$ is connected and that the mean
curvature of $\partial N$ with respect to the inward pointing normal is everywhere $\ge h$, where $h>0$.
Let $m=\dim(N)-1$ and suppose $V$ is an $m$-dimensional 
varifold in $N$ and that
\begin{equation*}
   \delta V(X) +  h \int |X|\, d\mu_V \ge 0
\end{equation*}
for every compactly supported $C^1$ vectorfield on $N$ such that
\begin{equation*}
  \text{$X \cdot \nu_N \ge 0$ at all points of $\partial N$.}
\end{equation*}
\begin{enumerate}
\item\label{strongI} If $\spt V$ contains any point of $\partial N$, then it must contain all of $\partial N$
and $\partial N$ must have constant mean curvature $h$.  
\item\label{strongII}
If $V$ is a rectifiable integral varifold with mean curvature $\le h$, 
then $V$ can be written as $W+W'$ where the support of $W$ is $\partial N$
and the support of $W'$ is disjoint from $\partial N$.
\end{enumerate}
\end{theorem}
 
The proof is similar to the proof of Theorem~\ref{SolomonWhite}, except that one uses 
Theorem~\ref{main-boundedversion} in place of Theorem~\ref{main}.

\newcommand{\hide}[1]{}


\begin{bibdiv}

\begin{biblist}


\bib{AllardFirstVariation}{article}{
  author={Allard, William K.},
  title={On the first variation of a varifold},
  journal={Ann. of Math. (2)},
  volume={95},
  date={1972},
  pages={417--491},
  issn={0003-486X},
  review={\MR {0307015},
  Zbl 0252.49028.}}
  \hide{(46 \#6136)}
  
  \bib{DierkesSchwab}{article}{
   author={Dierkes, Ulrich},
   author={Schwab, Dirk},
   title={Maximum principles for submanifolds of arbitrary codimension and
   bounded mean curvature},
   journal={Calc. Var. Partial Differential Equations},
   volume={22},
   date={2005},
   number={2},
   pages={173--184},
   issn={0944-2669},
   review={\MR{2106766.}},
}  \hide{ (2005h:35091)}

  
\bib{TomMax}{article}{
   author={Ilmanen, T.},
   title={A strong maximum principle for singular minimal hypersurfaces},
   journal={Calc. Var. Partial Differential Equations},
   volume={4},
   date={1996},
   number={5},
   pages={443--467},
   issn={0944-2669},
   review={\MR{1402732.}},
} \hide{(97g:49052)}

\bib{JorgeTomi}{article}{
   author={Jorge, Luqu{\'e}sio P.},
   author={Tomi, Friedrich},
   title={The barrier principle for minimal submanifolds of arbitrary
   codimension},
   journal={Ann. Global Anal. Geom.},
   volume={24},
   date={2003},
   number={3},
   pages={261--267},
   issn={0232-704X},
   review={\MR{1996769.}},
} \hide{(2004f:53063)}


\bib{Schaetzle}{article}{
   author={Sch{\"a}tzle, Reiner},
   title={Quadratic tilt-excess decay and strong maximum principle for
   varifolds},
   journal={Ann. Sc. Norm. Super. Pisa Cl. Sci. (5)},
   volume={3},
   date={2004},
   number={1},
   pages={171--231},
   issn={0391-173X},
   review={\MR{2064971.}},
} \hide{(2005e:49084)}

\bib{SimonBook}{book}{
  author={Simon, Leon},
  title={Lectures on geometric measure theory},
  series={Proceedings of the Centre for Mathematical Analysis, Australian National University},
  volume={3},
  publisher={Australian National University Centre for Mathematical Analysis},
  place={Canberra},
  date={1983},
  pages={vii+272},
  isbn={0-86784-429-9},
  review={\MR {756417},
  Zbl 0546.49019.}
}  \hide{ (87a:49001)}
		
		
\bib{SolomonWhite}{article}{
  author={Solomon, Bruce},
  author={White, Brian},
  title={A strong maximum principle for varifolds 
  that are stationary with respect to even parametric elliptic functionals},
  journal={Indiana Univ. Math. J.},
  volume={38},
  date={1989},
  number={3},
  pages={683--691},
  issn={0022-2518},
  review={\MR {1017330},
  Zbl 0711.49059.} 
}  \hide{ (90i:49052)} 


\bib{WhiteIsoperimetric}{article}{
   author={White, Brian},
   title={Which ambient spaces admit isoperimetric inequalities for
   submanifolds?},
   journal={J. Differential Geom.},
   volume={83},
   date={2009},
   number={1},
   pages={213--228},
   issn={0022-040X},
   review={\MR{2545035}},
}

\end{biblist}

\end{bibdiv}
\end{document}